\documentclass[a4paper,12pt,reqno]{amsart}

\usepackage{amssymb,amsmath,amsthm,graphics,amscd,amsfonts}
\theoremstyle{plain}
\newtheorem{theorem}{Theorem}[section]

\newtheorem{lemma}[theorem]{Lemma}

\newtheorem{remark}[theorem]{Remark}

\def\a{\alpha}

\def\b{\beta}
\def\bp{\overline{\partial}}

\def\o{\omega}
\def\O{\Omega}

\def\p{\partial}

\def\.{\cdot}

\def\beq{\begin{equation}}
\def\eeq{\end{equation}}
\def\bea{\begin{eqnarray*}}
\def\eea{\end{eqnarray*}}
\def\ba{\begin{array}}
\def\ea{\end{array}}

\def\z{\zeta}

\def\r{\end{proof}}

\def\wt{\widetilde}

\def \ZM{\mathbb{Z}}
\def \CM{\mathbb{C}}

\newcommand{\bfC}{{\mathbb C}}

\newcommand{\barj}{{\overline j}}

\newcommand{\barz}{{\overline z}}

\newcommand{\barpartial}{{\overline \partial}}

\newcommand{\mapright}[1]{\smash{\mathop{   \hbox to 0.7cm{\rightarrowfill}}
  \limits^{#1}}}

\def\tr{{\hbox{tr}}}

\def\div{\mathrm{div}}

\def\Ll{{\mathcal L}}

\def\tr{\mathrm{tr}}
\def\vol{\mathrm{vol}}

\def\6{\partial}

\title{An integral invariant from the view point of locally conformally K\"ahler geometry}
\author{Akito Futaki}
\author{Kota Hattori}
\author{Liviu Ornea}
\thanks{The third author thanks Tokyo Institute of Technology for hospitality during part of the
preparation of this work. }

\address{Department of Mathematics, Tokyo Institute of Technology, 2-12-1, O-okayama,
Meguro, Tokyo 152--8551, Japan} 
\email{futaki@math.titech.ac.jp}
\address{Department of Mathematics, Tokyo Institute of Technology, 2-12-1, O-okayama,
Meguro, Tokyo 152--8551, Japan} 
\email{Kota Hattori <hattori.k.ae@m.titech.ac.jp>}
\address{Univ. of Bucharest, Faculty of Mathematics,
14 Academiei str, 70109 Bucharest, Romania, and 
Institute of Mathematics ``Simion Stoilow" of the Romanian Academy, 21, Calea Grivitei str.,
010702-Bucharest, Romania.}
\email{lornea@gta.math.unibuc.ro, Liviu.Ornea@imar.ro}

\date{May 24, 2011}

\begin{document}
 
\begin{abstract} In this paper we study an integral invariant which obstructs
the existence on a compact complex manifold of a volume form with the determinant of its Ricci form proportional to itself, in particular obstructs the
existence of a K\"ahler-Einstein metric, and has been studied since 1980's. We study this invariant from the view point of locally conformally K\"ahler
geometry. We first see that we can define an integral invariant for coverings of compact complex manifolds with automorphic volume forms.
This situation typically occurs for locally conformally K\"ahler manifolds. 
Secondly, we see that this invariant coincides with the former one. We also show that the invariant
vanishes for any compact Vaisman manifolds.
 \end{abstract}
\maketitle

\section{Introduction}

In \cite{futaki83.1}, the first author introduced an integral invariant defined on Fano manifolds and showed that it obstructs the existence of K\"ahler-Einstein
metrics. More precisely, if $M$ is a Fano manifold of dimension $n$ and 
$$\omega = i\, g_{i\barj}\, dz^i \wedge d\barz^j$$ 
is a K\"ahler form representing $2\pi c_1(M)$, there exists a real smooth function 
$F \in C^\infty(M)$ such that the Ricci form 
$$\rho(\omega) = - i \partial \barpartial \log \det g$$ is written as 
$$\rho(\omega) - \omega = i \partial \barpartial F$$
since both $\rho(\omega)$ and $\omega$ represent $2\pi c_1(M)$.
Then the invariant is defined as a character $f$ of the Lie algebra $\mathfrak h(M)$ of all holomorphic vector fields on $M$ into $\mathbb C$ 
and is 
expressed for $X \in \mathfrak h(M)$ by
\begin{equation}\label{invariant1} f(X) = \int_M\, XF\, \omega^n.
\end{equation}
This invariant was later extended in various ways. We first briefly review the various extension. 

The first line of extension is as invariants for compact K\"ahler manifolds with fixed K\"ahler class. Let $(M, [\omega])$ be a compact
K\"ahler manifold $M$ with a K\"ahler class $[\omega]$. Then we can extend $f$ as an obstruction to the existence of a K\"ahler form in $[\omega]$
of constant scalar curvature (\cite{futaki83.2}, \cite{calabi85}). This is defined by the same formula (\ref{invariant1}) if we replace the condition of $F$ 
by
$$ \sigma - \int_M \sigma \omega^n/\vol(M) = \Delta F$$
where $\sigma$ denotes the scalar curvature of $\omega$. When $[\omega]$ is an integral class this was further reformulated by Donaldson 
\cite{donaldson02} as an invariant for polarized schemes, and was used to
define the notion of K-stability. In a guise the reformulated invariant was expressed as slopes for subschemes by Ross and Thomas \cite{rossthomas06}. 
For Fano manifolds with the anticanonical class, 
the invariant $f$ has recently been extended to an obstruction to the existence of K\"ahler-Einstein metrics with cone singularities 
along a divisor (Donaldson \cite{donaldson1102}, 
Li \cite{ChiLi1104}), and it is used to define logarithmic K-stability. 
Around the same time as the work \cite{futaki83.2} and \cite{calabi85}, the invariant $f$ obstructing the constant scalar curvature K\"ahler metric was extended further by Bando \cite{bando83} 
to a family of invariants $f_k$, $k=1, \cdots, n$, where $f_k$ obstructs the existence of a K\"ahler form in $[\omega]$ 
such that the $k$-th Chern form $c_k(\omega)$ is harmonic. Notice that the scalar curvature is constant if and only if the first Chern form is 
harmonic by the second Bianchi identity. Thus $f_1$ coincides with $f$. Bando's idea can be further extended to transverse K\"ahler geometry of 
compact Sasaki manifolds \cite{FOW}.

The second line of extension was obtained in \cite{futakimorita85}, but this extension is obtained by relating the invariant $f$ for Fano manifolds
to invariants classically known in the theory of the equivariant cohomology. Again, let $M$ be a Fano manifold and 
$\omega = i\, g_{i\barj}\, dz^i \wedge d\barz^j$
is a K\"ahler form representing $2\pi c_1(M)$. 
By the solution by Yau  \cite{yau77} to the Calabi conjecture, there exists a K\"ahler form $\eta$ representing $2\pi c_1(M)$ such that $\rho(\eta) = \omega$. 
Then we can rewrite $f$ as in (\ref{invariant1}) in terms of $\eta$ and obtain
\beq\label{invariant2}
f(X) = \int_M \div X \, \rho(\eta)^n
\eeq
where
\beq\label{invariant3}
\div X\cdot\eta^n = \p i(X) \eta^n
\eeq
and $i(X)$ denotes the interior product by $X$, see \cite{futakimorita85} or \cite{futaki88}  for the detail. Note that, instead of the K\"ahler form $\eta$, we may 
use any volume form $\Omega$ and its Ricci form 
$$\rho(\Omega) = - i \p\bp \log \O.$$ 
Then we may
write (\ref{invariant2}) as
\beq\label{invariant4}
f(X) = \int_M \div X \, \rho(\O)^n
\eeq
where
\beq\label{invariant5}
\div X\cdot\O = \p i(X) \O.
\eeq
We can prove that $f$ is then independent of the choice of $\Omega$, and thus we do not need to assume
that $M$ is Fano or K\"ahler. Thus we obtained an invariant for (possibly non-K\"ahler) complex manifolds. This last invariant is the one we wish
to study in this paper. Note also that we can rewrite (\ref{invariant4}) as
\beq\label{invariant6}
f(X) = - \int_M X \left(\frac{\rho(\Omega)^n}{\Omega}\right)\, \Omega.
\eeq
Therefore the invariant $f$ is an obstruction to the existence of a volume form $\O$ such that $\rho(\O)^n/\O$ is constant.

We remark in passing that there is a larger family of invariants including these two lines of extension (\cite{futaki04-1}). Among them we have a family
of invariants which obstructs asymptotic Chow semistability of polarized manifolds (\cite{futaki04-1}, \cite{FOS08}). By computing them for a 7-dimensional toric Fano manifold
suggested by Nill and Paffenholz \cite{NillPaffen}, Ono, Sano and Yotsutani \cite{OSY09} showed that there is a K\"ahler-Einstein Fano manifold which is asymptotically unstable.

Now let us turn to the study of the invariant defined by (\ref{invariant4}) or (\ref{invariant6}). 
Let $M$ be a compact connected complex manifold of dimension $n$. Consider a covering space $\pi : \wt M \to M$ with the group
$\Gamma$ of the deck transformations, and let $\chi : \Gamma \to \mathbb R^+$ be a homomorphism.
A volume form $\Omega$ on $\wt M$ is said to be automorphic with respect to $\chi$ if, for any $\gamma \in \Gamma$,  
\begin{equation}\label{Intro1}
\gamma^\ast \Omega = \chi(\gamma)\Omega.
\end{equation}
Such a covering with automorphic volume form naturally occurs for locally conformally K\"ahler manifolds as we shall see in the next section.

Given such a covering $\wt M$ with automorphic volume form with respect to $\chi$ we have a Ricci form $\rho_\O$ of $\O$ defined on $\wt M$ by
\beq\label{Intro2}
\rho_\O = - i \p\bp \log \O.
\eeq
Since $\O$ is automorphic, $\rho_\O$ is invariant under the action of $\Gamma$, and thus descends to a $2$-form on $M$ which 
is denoted by the same notation $\rho_\O$. This represents the first Chern class $2\pi c_1(M)$, and also $2\pi c_1(\wt M)$ upstairs if $\wt M$ is compact.

Denote by $\mathfrak h(M)$ and $\mathfrak h(\wt M)$ the Lie algebras of all holomorphic vector fields on $M$ and $\wt M$ respectively.
Denote also by $\mathfrak h_\Gamma (\wt M)$ the Lie subalgebra of $\mathfrak h(\wt M)$ consisting of all holomorphic vector fields on $\wt M$ which
are invariant under the action of $\Gamma$. Then a holomorphic vector field in $\mathfrak h_\Gamma(\wt M)$ descends to a
holomorphic vector field on $M$, and thus $\mathfrak h_\Gamma(\wt M)$ can be naturally regarded as a Lie subalgebra of $\mathfrak h(M)$.
For an $X$ in $\mathfrak h_\Gamma(\wt M)$, its divergence $\div X$ is defined by 
\beq\label{Intro3}
\div X\cdot\O = \p i(X) \O
\eeq
where $i(X)$ denotes the interior product by $X$. Since $\O$ is automorphic and $X$ is invariant under $\Gamma$ it follows that
$\div X$ is invariant under $\Gamma$ and that $\div X$ descends to a smooth function on $M$.

We define a linear function $f : \mathfrak h_\Gamma(\wt M) \to \bfC$ by
\beq\label{Intro4}
f(X) = \int_M \div X \, \rho_\O^n.
\eeq

The main theorem of this paper is the following.
\begin{theorem}\label{Intro5}\ \ $\mathrm{(a)}$\ \ 
Let $M$ be a compact connected complex manifold and $\wt M$ its covering space with the group $\Gamma$ of deck transformations.
Suppose that we are given a character $\chi : \Gamma \to \mathbb R^+$. Then $f$ is independent of the choice of the volume form 
automorphic with respect to $\chi$.\\
$\mathrm{(b)}$\ \ The invariants defined by (\ref{invariant4}) and (\ref{Intro4}) coincide.\\
\end{theorem}

A locally conformally K\"ahler manifold $(M,J,g)$ is said to be a Vaisman manifold if there is a metric in the conformal class of $g$ for which
the Lee form is parallel, see section 2 for more detail.

\begin{theorem}\label{Intro6}\ \ 
The invariant in the previous theorem vanishes for any compact Vaisman manifold.
\end{theorem}

This paper is organized as follows. In section 2 we summarize the basics of locally conformally K\"ahler geometry, and
give a proof of Theorem \ref{Intro6}. In section 3 we give a proof of Theorem \ref{Intro5}. In section 4 we compute the 
invariant for the one
point blow-up of the Hopf surface, and see that this surface gives an example of nontrivial invariant.

\section{Locally conformally K\"ahler manifolds}

Let $(M,J)$ be a connected complex manifold of complex dimension $n \ge 2$ with $J$ a complex structure.
A locally conformally K\"ahler structure (LCK structure for short) on $(M,J)$ is a covering 
$$ \Gamma \to (\widetilde M,\wt J,\wt \o) \to (M,J)$$
where $\widetilde M$ is a covering space of $M$, $\wt \o$ a K\"ahler form on $\widetilde M$, and $\Gamma$ the group of
deck transformations acting on $\widetilde M$ as holomorphic homotheties. Thus there is a homomorphism $\chi : \Gamma \to \mathbb R^+$
satisfying
$$  \gamma^\ast \wt \o = \chi(\gamma) \wt \o.$$
A $p$-form $\alpha$ on $\widetilde M$ ia said to be automorphic if $\gamma^\ast \alpha = \lambda(\gamma) \alpha$ for some character $\lambda : \Gamma \to \mathbb R^+$. The above K\"ahler form $\wt \o$ is an example of an automorphic $2$-form.

There is an equivalent definition of an LCK structure described as follows. An LCK structure is a collection of an open covering $M = \cup_{\alpha \in \Lambda} U_\alpha$ and
K\"ahler metrics $g_\alpha$ on $U_\alpha$ satisfying
$$ g_\alpha = c_{\alpha\beta} g_\beta$$
on $U_\alpha \cap U_\beta$ with $c_{\alpha\beta} \in \mathbb R_+$. Then $\{c_{\a\b}\}$ gives a cocycle. Let $\theta$ be a representative as a closed one form
defining the same cohomology class as $\{\log c_{\a\b}\}$. Thus we have $d\theta = 0$, and locally $\theta|_{U_\a} = df_\a$ for a smooth function $f_\a$ on $U_\a$ 
with $f_\b - f_\a = \log c_{\a\b}$ 
and $e^{f_\a}g_\a = e^{f_\b} g_\b$ on
$U_\a \cap U_\b$. Therefore $g := e^{f_\a}g_\a$ defines a global Hermitian metric locally conformal to a K\"ahler metric. The $1$-form $\theta$ is called the Lee form. 
Let  $\omega$ be the fundamental $2$-form defined by
 $$ \omega(X,Y) = g(JX,Y).$$
 Then one easily shows that
 \begin{eqnarray}
 d\omega &=& \theta \wedge\omega \label{lck1}\\
 d\theta &=& 0. \label{lck2}
 \end{eqnarray}
As an equivalent third definition we may
say that an Hermitian manifold $(M,J,g)$ is a locally conformally K\"ahler manifold if the fundamental $2$-form $\omega$ of $g$ 
satisfies (\ref{lck1}) and (\ref{lck2}). 

\begin{remark} When we say $(M,J,g)$ is an LCK manifold we assume that $\theta \ne 0$, that is, $(M,J,g)$ is not globally K\"ahler. 
\end{remark}

The equivalence between the second and the third definitions is obvious. 
To see that first implies the third, suppose that we are in a situation of the first definition. Let $L$ be the $\mathbb R$-bundle
given by $\wt M \times_\chi \mathbb R$. Since $\chi$ is $\mathbb R^+$-valued, $L$ is oriented and thus is a trivial bundle. 
It follows that $L$ has a nowhere zero section which defines a positive $\chi$-equivariant function $\phi$ on $\wt M$.
Then $\omega := \phi^{-1}\wt\omega$ is a $\Gamma$-invariant positive $2$-form. This $\omega$ satisfies the third definition with
$\theta = - \log \phi$.

We need only to show that the second implies the first. Suppose that we have K\"ahler forms $\omega_\a$ on $U_\a$ such that
$\omega_\a = c_{\a\b} \omega_\b$ with $c_{\a\b} \in \mathbb R^+$. Then $\{c_{\a\b}\} \in H^1(M, \mathbb R^{+\delta})$ defines a 
flat principal $\mathbb R^+$-bundle. The holonomy gives a character $\chi : \Gamma = \pi_1(M) \to \mathbb R^+$. 
Let $L = \wt M \times_\chi \mathbb R$ be the associated $\mathbb R$-bundle. 
We may regard $\{\omega_\a\}$ as a section of 
$$
L\otimes \Lambda^2T^{\ast} M = (\wt M \times_\chi \mathbb R)\otimes \Lambda^2T^\ast M = \wt M \times_\chi (\mathbb R \otimes p^\ast\Lambda^2T^{\ast} M)
$$
where $p : \wt M \to M$ is the projection.
Thus $\{\omega_\a\}$ defines a $\chi$-equivariant closed $2$-form $\wt \o$ on $\wt M$. This completes the equivalence of the three definitions of LCK structures.
 
 \begin{proof}[Proof of Theorem 2.1] Recall that 
a locally conformally K\"ahler manifold $(M,J,g)$ is said to be a Vaisman manifold if there is a metric in the conformal class of $g$ for which
the Lee form is parallel. It is shown in \cite{KamishimaOrnea05} that a Vaisman manifold is obtained as 
a quotient of the K\"ahler cone $C(S)$ of a Sasakian manifold $S$ 
by a subgroup $\Gamma$ of the homotheties acting freely and properly discontinuously. Then the proof follows from Lemma \ref{Ricci} below
since for the Reeb vector field $\xi = Jr\frac\p{\p r}$ on $C(S)$, $\xi - iJ\xi$ is a holomorphic flow and the Ricci tensor degenerates on this orbit.
See the arguments below for the notations.
\end{proof}

Recall that a Riemannian manifold $(S, g)$ of dimension $2m + 1$ is a Sasakian manifold if the cone $(C(S), \bar g) = (\mathbb R^+ \times S, 
dr^2 + r^2 g)$ is a K\"ahler manifold. Here $r$ is the standard coordinate on $\mathbb R^+$. The metric $\bar g$ is a warped product metric,
and the Riemannian geometry of $C(S)$ is easily studied from that of $S$.
Let $\bar\nabla$ and $\nabla$ be the Levi-Civita connections on $C(S)$ and $S$ respectively.
Let  $X, Y$ be tangent vector fields on $S$, which are naturally regarded as vector fields on $C(S)$ by the product 
structure $C(S) = \mathbb R^+ \times S$. Consider a vector field  $\Psi := r\frac \p{\p r}$ on $C(S)$. 
It is generally true for cone manifolds that 
\beq\label{Ricci1}
\bar\nabla_X \Psi = \bar\nabla_{\Psi} X  = X   
\eeq
and 
that 
\beq\label{Ricci2}
\bar\nabla_X Y = \nabla_X Y - g(X,Y) \Psi.
\eeq
Let $\bar R$ be the curvature tensors on $C(S)$.
Then using (\ref{Ricci1}) and (\ref{Ricci2}) we obtain
\begin{eqnarray*}
\bar R(X,\Psi, Y, \Psi) &=& \bar g (- \bar\nabla_X \bar\nabla_{\Psi} Y + \bar\nabla_{\Psi}\bar\nabla_X Y + \bar\nabla_{[X, \Psi]}Y, \Psi\\
&=& \bar g (- \bar\nabla_X Y + \bar\nabla_{\Psi}(\nabla_X Y - g(X,Y)\Psi), \Psi)\\
&=& - \bar g(\nabla_X Y - g(X,Y)\Psi, \Psi) + \bar g(\nabla_X Y - g(X,Y)\Psi, \Psi)\\
&=& 0.
\end{eqnarray*}
This implies that
\beq\label{Ricci3'}
\bar Ric(\Psi, \Psi) = 0
\eeq
where $\bar Ric$ denotes the Ricci tensor on $C(S)$. 
Let $J$ be the complex structure on $C(S)$. The vector field $\xi = Jr\frac\p{\p r}$ on $C(S)$ is called the Reeb vector field,
and it is a standard fact in Sasakian geometry that $\xi - iJ\xi$ is a holomorphic vector field.
Since the Ricci tensor on a K\"ahler manifold is $J$-invariant, (\ref{Ricci3'}) implies 
\beq\label{Ricci4'}
\bar Ric(\xi, \xi) = 0.
\eeq
From (\ref{Ricci3'}) and (\ref{Ricci4'}) we have proved the following.
\begin{lemma}\label{Ricci} The Ricci tensor on $C(S)$ vanishes on the plane spanned by $\xi$ and $J\xi = - r \frac \p{\p r}$.
\end{lemma}

For a Vaisman manifold $M$ we can also find an LCK metric $g$ for which the Ricci form $\rho(g)$ satisfies $\rho(g)^n = 0$,
showing that the integrand of (\ref{invariant4}) and (\ref{invariant6}) vanishes. Recall that the K\"ahler form $\wt \o$ on $C(S)$ is written as
$\wt \o = dd^c r^2$, see for example \cite{FOW}. As we have seen in Lemma \ref{Ricci}, we have $\rho(\wt \o)^{m+1} = 0$. Note that $n= m+1$ here. 
Then $\wt \o/r^2 = \frac 1 {r^2} dd^c r^2$ defines
a $\Gamma$-invariant 2-form and descends to $M$.  Since $dd^c \log r$ is the transverse K\"ahler form $\omega^T$  on the Sasaki manifold $S$
(but regarded as lifted to $C(S)$),
the Ricci form of $\wt \o/r^2$ is equal to $\rho(\wt \o) + 2(m+1)\omega^T$. This also degenerates on the orbit of the flow generated by $\xi - iJ\xi$. Hence
we have $\rho(\wt \o/r^2)^{m+1} = 0$ on the Vaisman manifold $M$. 

 \section{Proof of Theorem \ref{Intro5}}
In this section we prove the following result.

\begin{theorem}\label{result1}
Let $M$ be a compact connected complex manifold and $\wt M$ its covering space with the group $\Gamma$ of deck transformations.
Suppose that we are given a character $\chi : \Gamma \to \mathbb R^+$. Then $f$ defined by (\ref{Intro4}) 
is independent of the choice of the automorphic volume form $\Omega$ and its character $\chi$.
\end{theorem}

Theorem \ref{Intro5} follows from Theorem \ref{result1} because (a) in Theorem \ref{Intro5} is obtained by comparing $(\O_1, \chi)$ and
$(\O_2, \chi)$, and (b) in Theorem \ref{Intro5} is obtained by comparing $(\O_1, \chi)$ and
$(\O_2, 1)$.

 Let $M$ be a compact connected complex manifold and $\wt M$ be a covering space of $M$ with the group $\Gamma$ of deck transformations. 
 Let  $\O$ be a smooth volume form on $\wt M$ automorphic with respect to $\chi : \Gamma 
 \to \mathbb R^+$.
 If $z^1, \cdots, z^n$ are local holomorphic coordinates on $\wt M$, the volume form $\O$ can be expressed as
 \beq\label{proof1}
 \O = a\,idz^1 \wedge d\barz^1 \wedge \cdots \wedge i dz^n \wedge d\barz^n
 \eeq
where $a$ is a local positive smooth function. The Ricci form $\rho_\O$ and the divergence $\div X$ can be expressed using $a$ as
\beq\label{proof2}
\rho_\O = - i\p\bp \log a,
\eeq
and 
\beq\label{proof3}
\div X = \sum_{i=1}^n \frac{\p X^i}{\p z^i} + X \log a.
\eeq
From (\ref{proof3}) we obtain
\beq\label{proof4}
\bp \div X = i(X) \p\bp \log a.
\eeq
\begin{proof}[Proof of Theorem \ref{result1}]
Let $\O_0, \O_1$ be volume forms automorphic with respect to $\chi_0,\chi_1 : \Gamma \to \mathbb R^+$, respectively. Then $\O_i$ can be expressed as
\bea
\O_0 &=& a \,idz^1 \wedge d\barz^1 \wedge \cdots \wedge i dz^n \wedge d\barz^n,\\
\O_1 &=& \varphi a \,idz^1 \wedge d\barz^1 \wedge \cdots \wedge i dz^n \wedge d\barz^n,
\eea
where the positive real valued function $\varphi$ on $\wt M$ is given by $\O_1 = \varphi \O_0$. Then we have
\beq\label{proof6}
\gamma^*\varphi = \frac{\chi_1(\gamma)}{\chi_0(\gamma)} \varphi
\eeq
for all $\gamma\in \Gamma$.

Let $\O_t$ be
\beq\label{proof7}
\O_t = \varphi^t a \,idz^1 \wedge d\barz^1 \wedge \cdots \wedge i dz^n \wedge d\barz^n,
\eeq
for each $0 \le t \le 1$. Then each $\O_t$ is automorphic with respect to a character $\chi_t:=\chi_0^{1-t}\chi_1^t$. Thus a smooth family of linear maps $f_t:\mathfrak h_\Gamma (\wt M) \to \mathbb{C}$ is defined by
\beq\label{proof8}
f_t(X) = \int_M \div_t X \, \rho_{\O_t}^n,
\eeq
where $\div_t X$ is the divergence determined by $\O_t$. Then it suffices to show that
\beq\label{proof9}
\frac{d}{dt}f_t(X) = 0
\eeq
for all $X \in \mathfrak h_\Gamma (\wt M)$.

It is easy to see
\beq\label{proof10}
\frac d{dt} (\div_t X) = X(\log \varphi)
\eeq
and
\beq\label{proof11}
\frac d{dt} \rho_{\O_t} = - i\p\bp \log \varphi.
\eeq
Then we have
\bea
&& \frac d{dt} \int_M \div_t X \cdot \rho_{\O_t}^n \\
&=& \int_M X( \log \varphi ) \ \rho_{\O_t}^n 
- \int_M \div_t X\ i\p\bp ( \log \varphi )\ \wedge n\rho_{\O_t}^{n-1} \\
&=& \int_M  X( \log \varphi )\  \rho_{\O_t}^n 
+ \int_M \bp (\div_tX \ \wedge \p (i  \log \varphi )\ \wedge n\rho_{\O_t}^{n-1})\\
&& - \int_M \bp (\div_tX)\ \wedge \p (i  \log \varphi )\ \wedge n\rho_{\O_t}^{n-1}. 
\eea
Although $\varphi$ is not $\Gamma$-invariant, $\p \log \varphi$ is $\Gamma$-invariant from (\ref{proof6}) and descends to a $1$-form on $M$. Since $\div_tX$ and $\rho_{\O_t}$ are also defined globally on $M$, we can deduce
\beq\label{proof12}
\int_M \bp (\div_tX \ \wedge \p (i  \log \varphi )\ \wedge n\rho_{\O_t}^{n-1}) = 0
\eeq
from Stokes' Theorem. Therefore we have
\bea
&& \frac d{dt} \int_M \div_t X \cdot \rho_{\O_t}^n \\
&=& \int_M  X( \log \varphi )\  \rho_{\O_t}^n 
- \int_M \bp (\div_tX)\ \wedge \p (i  \log \varphi )\ \wedge n\rho_{\O_t}^{n-1}\\
&=& \int_M  X( \log \varphi )\  \rho_{\O_t}^n\\
&&  - \int_M (i(X)\p\bp \log (\varphi^t a))\ \wedge \p (i\log \varphi )\ \wedge n\rho_{\O_t}^{n-1} \\
&=& \int_M  X( \log \varphi )\  \rho_{\O_t}^n + \int_M  (i(X)\rho_{\O_t}^n)\ \wedge \p \log \varphi \\
&=& \int_M i(X) (\rho_{\O_t}^n \wedge \p \log \varphi  ) = 0
\eea
since $\rho_{\O_t}^n \wedge \p \log \varphi \equiv 0$ because of dimension reasons. This completes the proof of Theorem \ref{result1}.
\end{proof}

\begin{remark} In general 
the vanishing of $f$ is the obstruction to the existence of an automorphic  volume form on $\wt M$ with $\rho_\O = 0$.
But if $\wt M = M$ or $\chi$ is trivial, the vanishing of $f$ obstructs the existence of a volume form with
$\rho_\O^n = k\O$ for some constant $k$.
\end{remark}

\section{The localization formula and an example}

Now that the invariant is independent of the choice of $(\O, \chi)$ we may use an old result to compute the case when $\chi$ is
trivial. This is a residue formula for holomorphic vector fields.

Let $X$ be a holomorphic vector field in $\frak h (M)$. 
Define a section $L(X)$ of the endomorphism bundle $\mathrm{End}(TM)$ of the holomorphic tangent bundle $TM$ by
\beq\label{proof13}
L(X)Y = \nabla_XY - [X,Y].
\eeq
Suppose that the zero set $\mathrm{zero}(X)$ of $X$ consists of smooth submanifolds $\{Z_\lambda\}_{\lambda \in \Lambda}$. Then
$L(X)$ induces a section $L^\nu(X)$ of the endomorphism bundle of the normal bundle $\nu(Z_\lambda) = (TM|_{Z_\lambda})/TZ_{\lambda}$
of $Z_\lambda$. 

\begin{theorem}[Theorem 5.2.8 in \cite{futaki88}]\label{zero1}
If $L^\nu(X)$ is nonsingular at every $q \in \mathrm{zero}(X)$, we have the following localization formula
\beq
(\frac 1{2\pi})^n (n+1)f(X) =  \sum_\lambda \int_{Z_\lambda}((\div X + c_1(M))^{n+1}|_{Z_\lambda})/\det (L^\nu(X) + \frac i{2\pi}K) \nonumber
\eeq
where $K$ is the curvature form of $\nu(Z_\lambda)$ with respect to the induced Hermitian connection.
\end{theorem}

We provide an explicit computation of non-zero invariant on the blow-up at a point of a Hopf surface which by \cite{tr, vu} is an LCK manifold, using the localization formula. 

Let $H^2$ be a Hopf surface that we regard as $\CM^2\setminus \{0\}/\ZM$, where $\ZM$ is generated by the transformation $(z_1,z_2)\mapsto (2z_1,2z_2)$. We choose the fundamental domain on $\CM^2\setminus \{0\}$ to be  $\{(z_1,z_2)\ |\ 1 \leq |z_1|^2+|z_2|^2\leq 2\}$.

Let $M$ be the blow-up of $H^2$ at the point $(0, \frac 32)$. It will be convenient to change the coordinates $(z_1,z_2)$ into $(w_1,w_2)$ by:
$$w_1=z_1,\quad w_2=z_2-\frac 32.$$
Then the blow-up takes place at the origin $(w_1,w_2)=(0,0)$ and the  exceptional divisor $E$ is $\{(w_1:w_2)\}\cong \CM P^1\subset M$.

Let $X=z_1\frac \p{\p z_1}$ be the radial (global) vector field on $\CM^2$. Its zero set contains $(0,\frac 32)$. We shall equally denote by $X$ its lift to $M$. 
Its zero set on $M$ will certainly contain $\{z_1=0\}$, but also some other point that we now determine.

Take first local coordinates on $M$ around $(1:0)\in E$ to be 
$$\z_1=w_1, \quad \z_2=\frac{w_2}{w_1}.$$
This change of coordinates is consistent with the coordinates on the exceptional divisor. In the new coordinates, $X$ is written as
$$X=\z_1\frac \p{\p \z_1}-\z_2\frac \p{\p \z_2}.$$
In these coordinates $(\z_1,\z_2)$, $\mathrm{zero}(X)=\left\{ (0,0)\right\}.$ Hence, the zero is on $E$ (as $\z_1=0$). On the other hand, $\z_2=0$ implies $w_2=0$. Thus, the isolated zero of $X$ is $(w_1:w_2)=(1:0)$.

Now recall that, in general, if a holomorphic vector field $Y=a\frac \p{\p\z_1}+b\frac \p{\p\z_2}$, then
\begin{equation}\label{zero}
 L(Y)(\frac \p{\p\z_j})|_{\mathrm{zero}(Y)}=-\Ll_Y \frac \p{\p\z_j}+\nabla_X\frac \p{\p\z_j}=\frac{\p a}{\p\z_j}\frac \p{\p\z_1}+\frac{\p b}{\p\z_j}\frac \p{\p\z_2},
\end{equation}
as the $\nabla_Y=0$ on the zero set of $Y$.

Hence, in our case, for the point $(1:0)$, the localization formula reduces to:
$$\frac{\tr(L(X))^3}{\det(L^\nu(X))}=\frac{\left(\tr\begin{pmatrix}1&0\\0&-1\end{pmatrix}\right)^3}{\det \begin{pmatrix}1&0\\0&-1\end{pmatrix}}=0,$$
as the normal bundle of the point equals the tangent bundle at the point, and this is trivial, hence $\Theta=0$. 

So, the isolated zero does not contribute to the value of the invariant.

For the dimension 1 component of $\mathrm{zero}(X)$, take on $M$, around $(0:1)$, the coordinates:
$$\z_1=w_2,\quad \z_2=\frac{w_1}{w_2}.$$
In these coordinates $X$ takes the form
$$X=\z_2\frac \p {\p\z_2},$$
and $\mathrm{zero}(X)=\left\{(\z_1,0)\right\}$, a line represented by $(0:1)$. It is the proper transform of $\{z_1=0\}$. Using \eqref{zero}, we find now
$$L(X)=\begin{pmatrix}0&0\\0&1\end{pmatrix}.$$
The localization formula gives:
\begin{equation*}
 \int_Z\frac{\left(\tr \left( \begin{pmatrix}0&0\\0&1\end{pmatrix} +\frac{\sqrt{-1}}{2\pi}\Theta\right)\right)^3}{1+\frac{\sqrt{-1}}{2\pi}\Theta^\nu}=
\int_Z\frac{\left(1+c_1(Z)+c_1(\nu(Z))\right)^3}{1+c_1(\nu(Z))},
\end{equation*}
where $\nu(Z)$ is the normal bundle of the zero set $Z=\mathrm{zero}(X)$. 

Observe that $\nu(Z)=-[E]$. Indeed, if $\pi:M \to H^2$ denotes the natural projection, then:
$$[\pi(Z)]=[\pi(Z+E)],$$
and hence (as they are trivial line bundles),
$$0=\pi^*[\pi(Z)]=[Z+E]=[Z]+[E]=[\nu(Z)]+[E].$$
On the other hand, $c_1(Z)=0$, as $Z$ is an elliptic curve. We obtain:
\begin{eqnarray*}
 \int_Z\frac{\left(1+c_1(Z)+c_1(\nu(Z))\right)^3}{1+c_1(\nu(Z))}
&=& \int_Z(1-c_1([E]))^3 (1+c_1([E])))\\
&=& \int_Z(-3c_1([E]))+c_1([E])))\\
&=& Z\cdot (-2[E]))=-2.
\end{eqnarray*}
In conclusion, $3(\frac 1{2\pi})^2f(X)=-2\neq 0$. 





\end{document}